\documentclass{article}
\usepackage{graphicx} 
\usepackage{amsmath,amsthm,amssymb,geometry,url}
\usepackage{eurosym, enumerate,comment}
\usepackage[utf8]{inputenc}
\usepackage{authblk}
\usepackage{xcolor}
\usepackage{soul}
\usepackage{float}

\newtheorem{theorem}{Theorem}[section]
\newtheorem{lemma}[theorem]{Lemma}
\newtheorem{proposition}[theorem]{Proposition}
\newtheorem{corollary}[theorem]{Corollary}
\theoremstyle{definition}
\newtheorem{definition}[theorem]{Definition}
\theoremstyle{remark}
\newtheorem{remark}[theorem]{Remark}

\newtheorem{example}[theorem]{Example}
\newtheorem{question}[theorem]{Question}

\begin{document}

\title{A universal theory of switching for combinatorial objects, and applications to complex Hadamard matrices}
\date{}
\author[1]{Dean Crnkovi\'{c}}
\author[2]{Ronan Egan}
\author[1]{Andrea \v{S}vob}
\affil[1]{\small
Faculty of Mathematics, University of Rijeka, Croatia. Emails: deanc@math.uniri.hr; asvob@math.uniri.hr}
\affil[2]{\small School of Mathematical Sciences,
Dublin City University, Ireland. Email: ronan.egan@dcu.ie}
\maketitle

\begin{abstract}
The concept of switching has arisen in several different areas within combinatorics. The act of switching usually transforms a combinatorial object into a non-isomorphic object of the same type, in a way that some key property is preserved. Godsil-McKay switching of graphs preserves the spectrum, switching of designs preserves their parameters, and switching of binary codes preserves the minimum distance. For Hadamard matrices, the switching techniques introduced by Orrick proved to be an incredibly powerful tool in generating inequivalent Hadamard matrices. In this paper, we introduce a universal definition of switching that can be adapted to incorporate these known types of switching. Through this language, we extend Orrick's methods to Butson Hadamard and complex Hadamard matrices. We introduce switchings of these matrices that can be used to construct new, inequivalent matrices. We also consider the concept of trades in complex Hadamard matrices in this terminology, and address an open problem on the permissible size of a trade.
\end{abstract}

{\bf 2020 Mathematics Subject Classification: 05B05, 05B20}

{\bf Keywords:} Switching of combinatorial objects; complex Hadamard matrices.

\section{Introduction}

Many interesting objects in combinatorics are represented by matrices. Examples include graphs, block designs, Hadamard matrices and weighing matrices, to name a few. In these cases there is typically a notion of equivalence or isomorphism between two objects which is most apparent from the relationship between their representative matrices. For example, two graphs with adjacency matrices $A$ and $B$ are isomorphic if there exists a permutation matrix $P$ such that $PAP^{\top} = B$. In many cases, motivated either by classification problems or by constructing new objects retaining some specified property, a concept known as switching has developed, which is a means to convert one matrix into another with the same required property, by virtue of replacing some subset of the entries of the matrix with other allowable entries. Perhaps the best known example of this is Godsil-McKay switching of graphs \cite{GM-switching}, building on the initial concept of Seidel \cite{Seidel}, which converts the adjacency matrix of a graph into another with the same spectrum. Another switching method for graphs that preserves the spectrum was introduced in \cite{WQH} by Wang, Qiu and Hu. Some of the current authors recently extended the notion of switching to $2$-designs \cite{CS-switching} which preserves the parameters of the design. This builds on similar concepts used in the enumeration of certain symmetric designs by Denniston in \cite{Denniston}, and on the transformation of new symmetric or quasi-symmetric designs to new ones with the same parameters by Jungnickel and Tonchev in \cite{JunTon}. Switching is also a well known action on Latin squares, used to construct new ones from existing ones, see e.g., \cite{Wanless-Switching}. In coding theory, \"Osterg\aa rd adapted the concept of switching to construct non-isomorphic binary codes with the same parameters as a given code \cite{Ost-switching}. In the same paper, the enormous literature on switching methods in discrete mathematics is made apparent from the $106$ references therein, to which we refer the reader as a survey of which is beyond the scope of this work.

A similar concept was applied to Hadamard matrices by Orrick in \cite{OrrickSwitching} as a means to construct enormous quantities of inequivalent Hadamard matrices of order $32$, and to categorise Hadamard matrices according to a broader notion of equivalence incorporating these switching techniques. This was later found to be equivalent to Godsil-McKay switching \cite{AP-switching}. Orrick's switching of Hadamard matrices also motivated the concept of a trade in a complex Hadamard matrix introduced by \'{O} Cath\'{a}in and Wanless in \cite{POCWan}. We pay particular attention to these cases as the objects of interest in this paper are complex and Butson type Hadamard matrices. 

With many different but seemingly similar notions of switching equivalence now in the literature for different objects, it is appropriate to establish a unifying definition, capturing the notions that already exist and laying the framework for applications to combinatorial objects more generally. This is the first of two goals of this paper. The second is to develop switching methods for complex and Butson type Hadamard matrices that capture Orrick's switching of real Hadamard matrices, develop new methods, and resolve some open problems. Orrick's switching methods have been applied only to real Hadamard matrices. They have not been appropriately generalised to be applicable to Butson or complex Hadamard matrices more widely, but there are good reasons for doing so now. This can potentially produce several new equivalence classes of Butson matrices, and it is plausible that this could aid in an attempt to classify them at reasonable orders when combined with existence conditions. Butson and complex Hadamard matrices are gathering increasing interest as they are important objects in quantum information theory, due to their relationship with mutually unbiased bases of $d$-dimension complex space $\mathbb{C}^{d}$ and their applications to quantum entanglement \cite{CGGZ}. Another reason for this generalisation to complex matrices is that since a type of Orrick's switching of Hadamard matrices was shown to be equivalent to the Godsil-McKay switching method for producing cospectral graphs \cite{AP-switching,GM-switching}, it is reasonable to suspect that a generalisation of Hadamard switching could have wider implications, when seen as an example of switching under our universal definition.

The paper is outlined as follows. In the next section, we introduce the preliminary notation and terminology used throughout the paper. In Section \ref{sec:defs}, we introduce our general definitions that apply to switching of different objects and give some examples of their appearance in the literature. Further, we describe the background to Orrick's switching of Hadamard matrices from \cite{OrrickSwitching} which is the main motivation for this work, and which we build on in later in Section \ref{sec:generalize}. Section \ref{sec:general2} introduces a set of general conditions that are sufficient for switching complex Hadamard matrices, capturing the key aspects of the methods of Sections \ref{sec:Orrick} and \ref{sec:generalize} that make switching possible.

Finally, in Section \ref{sec:trades}, we recall the definition of a trade from \cite{POCWan}, consider it in terms of our definition, and answer an open question posed by the authors on the permissible size of a trade in a complex Hadamard matrix.

\section{Preliminaries}\label{sec:background}

Let $\mathbb{T}$ denote the set of all complex numbers of norm $1$, let $\zeta_{k} = e^{\frac{2 \pi \sqrt{-1}}{k}}$ be a primitive $k^{\rm th}$ root of unity, and let $\langle \zeta_{k} \rangle$ be the set of all $k^{\rm th}$ roots of unity generated multiplicatively by $\zeta_{k}$. A $n \times n$ matrix $H$ with entries in $\mathbb{T}$ is a complex \emph{Hadamard matrix} of order $n$ if $HH^{\ast} = nI_{n}$, where $H^{\ast}$ is the conjugate transpose of $H$, and $I_{n}$ is the $n \times n$ identity matrix. Put succinctly, $H$ is complex Hadamard if it is orthogonal and has entries of norm $1$. A complex Hadamard matrix of order $n$ exists for all $n$; take for example, the Fourier matrix $F_{n} = [\zeta_{n}^{(i-1)(j-1)}]_{1 \leq i,j \leq n}$. Interesting questions, both from the perspective of mathematical curiosity, and from that of applications to science, engineering, and the wider field of mathematics, arise from placing restrictions on the entries of $H$. If the entries are all elements of $\langle \zeta_{k} \rangle$, then we say that $H$ is a \emph{Butson Hadamard} matrix. There are finitely many of these at any order, and for a fixed $k$ existence is not guaranteed for all $n$, and it is often prohibited. The set of all $n \times n$ Butson Hadamard matrices over $\langle \zeta_{k} \rangle$ will be denoted by $\mathrm{BH}(n,k)$. The Fourier matrix $F_{k}$ is an example of an element of $\mathrm{BH}(k,k)$, but the existence question becomes much more difficult, in general. The most commonly studied set of Butson matrices are the real Hadamard matrices, i.e., $\mathrm{BH}(n,2)$. The term \emph{Hadamard matrix} alone typically refers to this case. The literature on Hadamard matrices is enormous, thanks in part to the long standing Hadamard conjecture, which demonstrates how difficult the question of existence can be. It is well known that a real Hadamard matrix of order $n > 2$ can exist only when $n \equiv 0 \mod 4$. The Hadamard conjecture posits that this condition is sufficient, but despite the efforts of many over a century, and numerous constructions exploiting number theory, geometry, algebra and combinatorics, the conjecture remains wide open. When $k > 2$, existence conditions are often more technical but less restrictive, and this makes forming a comparable conjecture more difficult. As such the problem is less well studied, and so there is great scope to find and construct new examples.


The purpose of this paper is not to dwell on the question of existence. The other most common problems of interest are of classification and of construction of matrices inequivalent to those already known. Let $H$ be a complex Hadamard Matrix. The orthogonality of $H$ is preserved by any permutation of the rows or columns and by any scalar multiplication of rows or columns by an element of $\mathbb{T}$. Two Hadamard matrices $H$ and $K$ that differ from one another by some sequence of these operations are called \emph{Hadamard equivalent}. This can be rephrased in terms of matrix algebra. Let $\mathrm{Mon}_{n}(\mathbb{T})$ denote the set of all monomial matrices with non-zero entries in $\mathbb{T}$. That is, a matrix $M \in \mathrm{Mon}_{n}(\mathbb{T})$ has exactly one non-zero entry, which has norm $1$, in each row and column. Permuting rows (columns) of $H$ and multiplication of rows (columns) of $H$ by a scalar is achieved by left (right) multiplication of $H$ by a matrix $M \in \mathrm{Mon}_{n}(\mathbb{T})$. The orbit of $H$ under the action of $\mathrm{Mon}_{n}(\mathbb{T})^{2}$ on $H$ defined by $H(M,N) = MHN^{\ast}$ is the \emph{equivalence class} of $H$. 

If to begin with we assume $H \in \mathrm{BH}(n,k)$, and our goal is to preserve the set of entries $\langle \zeta_{k} \rangle$, then we restrict to the action of $\mathrm{Mon}_{n}(\langle \zeta_{k} \rangle)^{2}$. Since the number of matrices $\mathrm{BH}(n,k)$ is finite, we can attempt to enumerate them. Considering the likely size of an equivalence class however, this becomes an enormous task very quickly. 

Much work has been done instead to classify matrices up to equivalence, at least for real Hadamard matrices. However, even up to equivalence, this problem quickly spirals out of the realm of computational feasibility, as the numbers of equivalence classes appear to grow exponentially fast. In the real case, the number of equivalence classes have long been classified completely at orders $4m$, for all $m \leq 7$. The orders are given in Table \ref{tab:2classes}.
\begin{center}
\begin{table}[H]
\footnotesize\centering
\begin{tabular}{|c|c|c|c|c|c|c|c|}
\hline
$m$ & 1 & 2 & 3 & 4 & 5 & 6 & 7 \\\hline
Classes & 1 & 1 & 1 & 5 & 3 & 60 & 487 \\\hline
\end{tabular}
\caption{Equivalence classes of real Hadamard matrices of order $4m$}\label{tab:2classes}
\end{table}
\end{center}
This evidence alone suggests that the number of classes grow rapidly with $n$. Orrick developed switching methods \cite{OrrickSwitching} to find new classes at orders $32$ and $36$, to enormous effect, which strongly motivates the work of this paper. Orrick's switching method preserves the Hadamard property, but does not necessarily produce matrices that are Hadamard equivalent. Repeated application of this method at a large computational scale yielded well in excess of three million classes at order $32$, and in excess of $18$ million at order $36$. However, if the notion of equivalence was extended so that two matrices obtained by switching were also considered equivalent, then Orrick found $11$ and $21$ classes at order $32$ and $36$, respectively. Even these numbers likely grow rapidly with $n$, but it seems that this is a more tractable way to enumerate the classes at larger orders. 
Later in \cite{Hadi32}, Kharaghani and Tayfeh-Rezaie completed a classification of Hadamard matrices of order $32$, and the total number of classes were found to be $13,710,027$. To our knowledge, Orrick's methods have not been applied yet to categorise these up to switching equivalence.

\section{Universal definition of switching}\label{sec:defs}

We introduce our main definitions here, which can be applied to a broad spectrum of matrices having some particular property that defines it. Throughout this section, let $\mathcal{A}$ be an alphabet and let $S = \mathrm{Sym}(\mathcal{A})$.

\begin{definition}\label{def:main}
Let $H$ be a $m \times n$ matrix having a set of properties $P$ with entries in an alphabet $\mathcal{A}$, with rows and columns partitioned into sets $R,R_{1},\ldots,R_{s}$ and $C,C_{1},\ldots,C_{t}$ respectively. Denote by $M_{ij}$ the submatrix defined by the rows of $R_{i}$ and columns of $C_{j}$. Extend the action of $S$ to $\mathcal{A}$-matrices entrywise. Suppose there exist $a_{ij} \in S$ for $1 \leq i \leq s$ and $1 \leq j \leq t$ such that the matrix $K$ obtained from $H$ by replacing the submatrices $M_{ij}$ with $a_{ij}(M_{ij})$ retains the properties of $P$, where at least one of $a_{i\ell}$ and $a_{kj}$ are distinct from $\mathrm{Id}(S)$ for any fixed $\ell$ and $k$. Then
\begin{enumerate}
\item the set of sets $\{R_{1},\ldots,R_{s}\}$ is a \emph{switching set of rows} of $H$ of \emph{size} $s$;
\item the set of sets $\{C_{1},\ldots,C_{t}\}$ is a \emph{switching set of columns} of $H$ of \emph{size} $t$;
\item the elements $a_{ij}$ are called \emph{switching coefficients};
\item acting on the submatrices $M_{ij}$ of $H$ with the switching coefficients $a_{ij}$ to obtain a new matrix $K$ with the properties of $P$ is called a \emph{switching} of $H$.
\end{enumerate}
\end{definition}

\begin{example}
Godsil-McKay switching \cite{GM-switching} preserves two relevant properties of a $\{0,1\}$-matrix: that it is the adjacency matrix of a graph, and its spectrum. The non-trivial switching coefficients act by interchanging $0$ and $1$, or alternatively, by adding $1$ if considering the entries of the matrix to be in $\mathbb{Z}_{2}$. 
\end{example}

In the context of Hadamard matrices, the alphabet $\mathcal{A}$ is either $\mathbb{T}$ or in the Butson case, $\langle \zeta_{k} \rangle$ for some $k$, and the action of $a_{ij} \in \mathrm{Sym}(\mathcal{A})$ is typically (but not necessarily) by scalar multiplication. The examples in this paper fall into this category, however we note that other actions can be reasonably applied, such as a Galois automorphism. We remark that more general interpretations can be necessary. For example, if $\mathcal{A} = \mathbb{Z}_{2}$ and $H$ is the incidence matrix of block design, then the action of the switching coefficients via addition of $0$ or $1$ can define a switching if the resulting matrix is the incidence matrix of a distinct block design with some preserved parameter. Restricting to scalar multiplication by either $0$ or $1$ would be too prohibitive in that case.

Definition \ref{def:main} is very general, even restricting to the Hadamard matrix interpretation, so some additional terminology will be required in order to avoid trivial or redundant cases. While they are in the context of the Hadamard interpretation, we note that they can be rewritten in terms of another object that may have a different notion of equivalence. Hereafter, let $H$ be a (complex) Hadamard matrix and let $\{R_{1},\ldots,R_{s}\}$ be a switching set of rows of $H$, let $\{C_{1},\ldots,C_{t}\}$ be a switching set of columns of $H$, and let $a_{ij} \in S = \mathrm{Sym}(\mathbb{T})$ be switching coefficients acting on $\mathbb{T}$.

\begin{definition}
If the matrix $K$ obtained by switching is necessarily Hadamard equivalent to $H$, then the switching sets are called \emph{degenerate}.
\end{definition}

\begin{definition}
If any subset of the switching coefficients that are not equal to $\mathrm{Id}(S)$ could be replaced with $\mathrm{Id}(S)$ and the resulting matrix $K$ is still Hadamard, then the switching sets are called \emph{reducible}. Otherwise they are called \emph{irreducible}.
\end{definition}

If for any distinct $\ell$ and $m$, the switching coefficients $a_{i\ell}$ and $a_{im}$ are equal for all $i$, then it is clear that $C_{\ell}$ and $C_{m}$ can be replaced with their union and the size of the switching set of columns decreases by one. Similarly, if for any distinct $\ell$ and $m$, the switching coefficients $a_{\ell i}$ and $a_{mi}$ are equal for all $i$, then it is clear that $R_{\ell}$ and $R_{m}$ can be replaced with their union and the size of the switching set of rows decreases by one. In this case the partition of $R_{\ell} \cup R_{m}$ into two sets is \emph{redundant}. 
An appropriate permutation of the rows and columns of $H$ ensures that the corresponding submatrices are joined. 

\begin{definition}\label{def:rank}
If an irreducible switching sets of rows and columns has no redundant partition, then we say the switching set is \emph{minimal}. The number of switching coefficients not equal to $\mathrm{Id}(S)$ of a minimal switching set is called the \emph{rank} of the switching set. If the switching set is rank $r$, then we call the act of switching a \emph{rank} $r$ \emph{switching}.
\end{definition}

If the rank is unbounded, then every Hadamard matrix of the same order would be equivalent under this definition. Also, the larger the rank, the less structure there is on the matrix, likely making it more difficult to repeatedly implement a switching in practice because of the added difficulty in locating the switching sets. It stands to reason then, that the interesting cases are low rank, non-degenerate switching sets. For real Hadamard matrices, what Orrick refers to as \emph{switching a closed quadruple} would be a rank $1$ switching by our definition. What is called \emph{switching a Hall set} is a rank $2$ switching. For the most part we focus on rank $1$ and $2$ switching hereafter.

\begin{example}
Large families of examples of switching sets for complex Hadamard matrices come from the complex analog of Bush type matrices recently defined in \cite{Hadi-Bush}. A matrix $H \in \mathrm{BH}(n^2,k)$ is of \emph{Bush-type} if it may be subdivided into $n \times n$ blocks $H_{ij}$ such that $JH_{ij} = H_{ij}J = \delta_{i,j}nJ$, for all $1 \leq i,j \leq n$, where $J$ is the $n \times n$ matrix of all ones. That is, each $n \times n$ block on the diagonal is $J$, and every other $n \times n$ block has constant row and column sum zero. It follows that any of the blocks equal to $J$ on the diagonal can be multiplied by any complex number of norm $1$ and orthogonality would be preserved, i.e., this is a rank $1$ switching of $H$. At least one Bush-type $\mathrm{BH}(n^{2},k)$ exists whenever a $\mathrm{BH}(n,k)$ exists, see \cite[Corollary 3.3]{Hadi-Bush}. 
\end{example}

\begin{remark}
The condition of square order of a Bush type Butson matrix is stronger than is necessary for a switching. We will call a matrix $H \in \mathrm{BH}(mn,k)$ a \emph{weak Bush-type} matrix if it may be subdivided into $m \times m$ blocks $H_{ij}$ such that $JH_{ij} = H_{ij}J = \delta_{i,j}mJ$, for all $1 \leq i,j \leq n$. We find rank $1$ switchings in weak Bush-type matrices similarly.
\end{remark}

Orrick showed that if Hadamard equivalence is extended to include his switching, the number of equivalence classes falls dramatically. On the other hand, if we include switchings of any rank, then conceivably all matrices are equivalent. The following question is then natural to ask.

\begin{question}
Is there a function $r(n)$ that sharply bounds above the minimum number $r$ such that every real Hadamard matrix of order $n$ must be equivalent up to the combination of Hadamard equivalence and switchings of rank up to $r$? More concretely, for a given $n$, what is the minimum number $r$ such that all Hadamard matrices of order $n$ are equivalent up Hadamard equivalence and rank up to $r$ switches? 
\end{question}

\subsection{Different types of switching in the literature}

Our focus is on Hadamard matrices here, but switching is a well known concept in discrete mathematics. The main inspiration for this work is the switching of Hadamard matrices described by Orrick in \cite{OrrickSwitching} which we discuss in more detail in the next section. In graph theory, switching is the process of obtaining a new graph by deleting some edges and introducing others. It was introduced by Seidel in \cite{Seidel} and described in terms of a subset $S \subseteq V$ of the vertices, where edges joining $x \in S$ to $y \not\in S$ are either deleted if they exist, or introduced if they do not already exist. In terms of its adjacency matrix, this means replacing some entries equal to $1$ or $0$ with $0$ or $1$ respectively. Suppose that the rows of the adjacency matrix of a graph $G$ are partitioned into the row set labelled by the vertices in the subset $S$ and its complement $V \setminus S$, which we call $R_{1}$ and $R_{2}$ respectively. Partition the columns similarly into $C_{1}$ and $C_{2}$. Then Seidel's switching is the application of the transpositions $\alpha_{12} = \alpha_{21} = (0,1) \in \mathrm{Sym}(\{0,1\})$ to the entries of the submatrices defined by $R_{1}$ and $C_{2}$, and by $R_{2}$ and $C_{1}$ simultaneously. This is a rank $2$ switching.

The most famous applications of switching in graphs are those such that the obtained graph and the original graph are cospectral, with the method of Godsil and McKay in \cite{GM-switching} being the most well known of these. The graph switching of \cite{WQH} retains the same property. In these cases, the spectrum is a property that we add to the set $P$, which is to be preserved by the switching. The property that the matrix is the adjacency matrix of a graph must be included, as an arbitrary switching of entries in $\{0,1\}$-matrix does not preserve this.

Inspired by graph switching, \"Osterg\aa rd \cite{Ost-switching} introduced switching to coding theory as a means to produce non-isomorphic binary codes with the same parameters, namely minimum distance. This is done by altering precisely two coordinates in every codeword where the corresponding pairs of entries are distinct. This can be achieved by a rank $1$ switching.

Another known use of switching is in design theory. Described in generality by Crnkovi\'{c} and \v{S}vob in \cite{CS-switching}, a switching can be applied to $2$-$(v,k,\lambda)$ design by deleting and introducing elements to the incidence structure $\mathcal{I} \subseteq \mathcal{P} \times \mathcal{B}$ where $\mathcal{P}$ and $\mathcal{B}$ are the sets of points and blocks of the design $\mathcal{D}$. In a manner similar to the Seidel switching of graphs, a switching set of $\mathcal{D}$ is a set of blocks $\mathcal{B}_{1}$ where incidences are either deleted or introduced between a specified subset of the points, and the blocks of $\mathcal{B}_{1}$, to produce a new $2$-$(v,k,\lambda)$ design $\mathcal{D}'$. If we label the columns of the incidence matrix with the blocks which are partitioned into $\mathcal{B}_{1}$ and $\mathcal{B} \setminus \mathcal{B}_{1}$, then the blocks of $\mathcal{B}_{1}$ label the switching set of columns $\mathcal{C}_{1}$. The point set $\mathcal{P}_{1}$ such that the incidence $(P,B) \in \mathcal{I} \Leftrightarrow (P,B) \not\in \mathcal{I}'$ labels the switching set of rows $R_{1}$. The application of the transposition $\alpha_{11} = (0,1) \in \mathrm{Sym}(\{0,1\})$ to the entries of the submatrix defined by $R_{1}$ and $C_{1}$ is a rank $1$ switching. 

Switching in designs appears much earlier in the literature in different terminology. Jungnickel and Tonchev \cite{JunTon} describe the transformation of symmetric or quasi-symmetric designs by means of interchanging incidences and non-incidences for all points on a maximal arc. This switching preserves the parameters of the design. What Denniston refers to as switching an oval in \cite{Denniston} is a similar technique, which was used to enumerate the symmetric $(25,9,3)$-designs.


The term switching is also used to describe an action on the entries of a Latin square used to produce distinct new Latin squares. Cycle switching refers to the swapping of two rows of a $2 \times k$ subrectangle of an $n \times n$ Latin square where $k < n$; see \cite{Wanless-Switching} for further background. This type of switching does not lend itself to a rank $1$ switching in the language of Definition \ref{def:rank}, as swapping two rows is not necessarily replicable by applying some $\alpha \in \mathrm{Sym}(n)$ to the submatrix entrywise. However, it is a rank $2$ switching where $R_{1}$ and $R_{2}$ are sets of size $1$, $C_{1}$ is a set of size $k$, and $\alpha_{11},\alpha_{21} \in \mathrm{Sym}(n)$ and possibly distinct.

\subsection{Orrick's switching of real Hadamard matrices}\label{sec:Orrick}

Orrick's versions of switching take two forms, one for Hadamard matrices of order $n \equiv 0 \mod 8$, called \emph{switching a closed quadruple}, and one for when $n \equiv 4 \mod 8$, called \emph{switching a Hall set}. Our goal in this section is to describe these operations in this more general setting. Later, we will present similar ideas that extend more widely to Butson type matrices. In this section, all Hadamard matrices are in $\mathrm{BH}(n,2)$, Hadamard equivalence is under the action of $\mathrm{Mon}_{n}(\langle -1 \rangle)^{2}$ on the set $\mathrm{BH}(n,2)$, and switching coefficients are either $1$ or $-1$, acting by mutliplication. We denote the row vector of all ones of length $m$ by $j_{m}$

\subsubsection{Switching a closed quadruple}

The origins of closed quadruples go back to earlier work of Kimura, but Orrick shows that finding a closed quadruple in a Hadamard matrix is essentially equivalent to using equivalence operations to convert the first four rows of a matrix of order $n = 4m$ into the form
\[
\left(\begin{array}{rrrr}
j_{m} & j_{m} & j_{m} & j_{m} \\
j_{m} & -j_{m} & j_{m} & -j_{m} \\
j_{m} & j_{m} & -j_{m} & -j_{m} \\
j_{m} & -j_{m} & -j_{m} & j_{m} \end{array}\right).
\]
It is not hard to show that any subsequent row of the form $(a \,~ b \,~ c \,~ d)$ with each of $a,b,c,d$ having length $m$, must have entries summing to $0$. Firstly, this implies $m$ is even so $n \equiv 0 \mod 8$. This also means the inner product $(j_{n},x) = 0$ for all $x \in \{a,b,c,d\}$. As a consequence, you can negate any of the $4 \times m$ blocks in these top four rows and the Hadamard property is retained. In the language of Definition \ref{def:main}, the rows of $H$ are partitioned into $R$ and $R_{1}$, where $R_{1}$ is the set of the first four rows and $R$ is the remaining set. The columns are partitioned into $C$ and $C_{1}$ where any one of the four aforementioned blocks of columns comprise $C_{1}$, and $C$ is the remaining set. Then $\{R_{1}\}$ and $\{C_{1}\}$ are switching sets of rows and columns of size $1$, respectively. The sole switching coefficient $a_{11} = -1$, and the submatrix $M_{11}$ is any one of the distinct $4 \times m$ submatrices read from left to right in the top four rows. This is a rank $1$ switching.

Of course one can negate any subset of the $4$ submatrices and retain the Hadamard property. This is just multiple rank $1$ switchings simultaneously. It could also be thought of as partitioning the columns into the sets $C,C_{1},\ldots,C_{t}$ for $1 \leq t \leq 4$, with $C$ being empty if $t = 4$. Then $\{R_{1}\}$ and $\{C_{1},\ldots,C_{t}\}$ are switching sets of rows and columns of size $1$ and $t$ respectively, with all switching coefficients equal to $-1$. However, this would be reducible, and it decomposes into rank $1$ switchings. Further, switching all four blocks would be degenerate, and switching any three blocks is Hadamard equivalent to switching the other one.

\subsubsection{Switching a Hall set}

The other switching method of Orrick is to put the Hadamard matrix of order $n = 4m+4$, if possible, into the following form:

\[
H = \left(\begin{array}{rrrr|rrr|rrr|rrr|rrr}
1 & - & - & - & 1 & \cdots & 1 & 1 & \cdots & 1 & 1 & \cdots & 1 & - & \cdots & - \\
- & 1 & - & - & 1 & \cdots & 1 & - & \cdots & - & - & \cdots & - & - & \cdots & - \\
- & - & 1 & - & 1 & \cdots & 1 & - & \cdots & - & 1 & \cdots & 1 & 1 & \cdots & 1 \\
- & - & - & 1 & 1 & \cdots & 1 & 1 & \cdots & 1 & - & \cdots & - & 1 & \cdots & 1 \\ \hline
1 & 1 & 1 & 1 & & & & & & & & & & & & \\
\vdots & \vdots & \vdots & \vdots & & A_{11} & & & A_{12} & & & A_{13} & & & A_{14} & \\
1 & 1 & 1 & 1 & & & & & & & & & & & & \\ \hline
- & 1 & 1 & - & & & & & & & & & & & & \\
\vdots & \vdots & \vdots & \vdots & & A_{21} & & & A_{22} & & & A_{23} & & & A_{24} & \\
- & 1 & 1 & - & & & & & & & & & & & & \\ \hline
- & 1 & - & 1 & & & & & & & & & & & & \\
\vdots & \vdots & \vdots & \vdots & & A_{31} & & & A_{32} & & & A_{33} & & & A_{34} & \\
- & 1 & - & 1 & & & & & & & & & & & & \\ \hline
1 & 1 & - & - & & & & & & & & & & & & \\
\vdots & \vdots & \vdots & \vdots & & A_{41} & & & A_{42} & & & A_{43} & & & A_{44} & \\
1 & 1 & - & - & & & & & & & & & & & &
\end{array}\right).
\]
Here each of the blocks $A_{ij}$ are $m \times m$, where $A_{ij}$ has row and column sum equal to $2$ if $i = j$, and equal to $0$ otherwise. Note that this means that $m$ must be even, or equivalently, the order of $H$ is $n \equiv 4 \mod 8$. If we negate $j^{\rm th}$ one of the $4 \times m$ and $m\times 4$ rank $1$ submatrices in the first row, and simultaneously negate the corresponding $m \times 4$ matrix, i.e., the $j^{\rm th}$, in the first column, then the Hadamard property is preserved. It is a little more work to verify that by putting the first four rows and columns of $H$ into this form through equivalence operations, the row and column sums for the submatrices $A_{ij}$ are as claimed.

Suppose for example that one negates the $4 \times m$ and $m \times 4$ all-ones submatrices in the first row and column. Then let $R_{1}$ denote the first four rows of $H$, $R_{2}$ denote the next $m$ rows, and $R$ denote the rest. Similarly, let $C_{1}$ denote the first four columns of $H$, $C_{2}$ denote the next $m$ columns, and $C$ denote the rest. Then $\{R_{1},R_{2}\}$ and $\{C_{1},C_{2}\}$ are switching sets of rows and columns of size $2$. The all-ones submatrices being negated are $M_{12}$ and $M_{21}$, hence there are two switching coefficients $a_{12}$ and $a_{21}$ equal to $-1$, while $a_{11} = a_{22} = 1$. Hence, this is a rank $2$ switching. As in the case of closed quadruples, any combination of the submatrices indicated can be negated, which can be described similarly as a larger but reducible switching set which decomposes into rank $2$ sets.

\subsection{Parameterisations of complex Hadamard matrices}\label{sec:paras}

Complex Hadamard matrices are often described as being members of a continuous family, perhaps most notably in the database \cite{karol} and in many of sources referred to therein, to which we refer the reader for a thorough background on the subject. A one-parameter example given by the matrix
\[
F_{4}(z) = \left(\begin{array}{rrrr} 
1 & 1 & 1 & 1 \\
1 & iz & -1 & -iz\\
1 & -1 & 1 & -1 \\
1 & -iz & -1 & iz\end{array}\right)
\]
is complex Hadamard where $i = \sqrt{-1}$ for any $z \in \mathbb{T}$. It is clear that altering $z$ in this example is a form of switching. In fact, in this example it can be described as a rank $1$ switching. The following one-parameter family given by Di\c{t}\u{a} in \cite{Dita},
\[
D_{6}(z) = \left(\begin{array}{rrrrrr} 
1 & 1 & 1 & 1 & 1 & 1\\
1 & -1 & i & -i & -i & i\\
1 & i & -1 & iz & -iz & -i\\
1 & -i & i\overline{z} & -1 & i & -i\overline{z}\\
1 & -i & -i\overline{z} & i & -1 & i\overline{z}\\
1 & i & -i & -iz & iz & -1 \end{array}\right),
\]
is also Hadamard no matter what choice of $z \in \mathbb{T}$ is made, but in this case, altering $z$ is a rank $2$-switching. The following formal definition is from \cite{CHGuide}.

\begin{definition}
An affine Hadamard family $H(\mathcal{R})$ stemming from a dephased $n \times n$ complex Hadamard matrix $H$ is the set of complex Hadamard matrices associated
with a subspace $\mathcal{R}$ of a space of all real $n \times n$ matrices with zeros in the first row and column, and
\[
H(\mathcal{R}) = \{H \circ \mathrm{EXP}(i \cdot R) \, \mid \,  R \in \mathcal{R}\},
\]
where $A \circ B$ denotes Hadamard product.

Parameterisations been have shown to be a powerful tool in the generation of new complex Hadamard matrices, and of Butson Hadamard matrices when the parameters are restricted to root of unity values. Affine families can account for multiple inequivalent Butson matrices. For example, in \cite{FerencQCH}, it is shown that members of the 15 equivalence classes in $\mathrm{BH}(8,4)$ can be described by three affine families.

It is apparent that altering the parameters in an affine family is a form of switching. Since we focus on complex Hadamard matrices in this paper, our examples of matrices containing switching sets can usually be said to admit a parameterisation. To illustrate that switching is more general than parametrization, we point out that in \cite{FerencQCH}, it is shown that each of the affine families contains the real Hadamard matrix of order 8, which is unique up to equivalence. Therefore, BH(8,4) matrices actually belongs to the same class up to the usual equivalence and rank 2 switchings. Switching can be more general, especially in the context of other objects like weighing matrices containing entries equal to zero, or when the switching coefficient does not act by multiplication. When switching is done by multiplying submatrices by a fixed number, this is often equivalent to altering a parameter in an affine family. Given the strong relationship between these two viewpoints, where relevant we will compare the examples and results in the following sections to known constructions of affine families of complex Hadamard matrices.
\end{definition}

\section{Switching conditions for complex Hadamard matrices}\label{sec:general2}

In this section, we describe some more general conditions that, if satisfied by a complex Hadamard matrix, enable a switching. 
Our intention in this section is to extract the key properties of the matrices in forms that admit switching, rather than to explicitly describe examples, which we will do in Section \ref{sec:generalize}. These general forms are also inspired by known switchings of designs and graphs.

\subsection{Conditions for rank $1$ switchings}

In \cite{CS-switching}, switching for 2-designs was introduced and applied to real Hadamard matrices related to 2-designs. This switching, translated to the language of real Hadamard matrices is given in Theorem \ref{des-had}.

\begin{theorem}\label{des-had}
Let $H$ be a real Hadamard matrix. Further, let $\{R,R_1\}$ be a partition of the set of rows of $H$ and let $\{ C, C_1 \}$ be a partition of the set of columns of $H$, such that

\begin{enumerate}
    \item the submatrix $H_1$ of $H$ determined by the rows from $R_1$ and the columns from $C$ has constant columns, i.e. $h_{rj}=h_{sj}$, for $r,s \in R_1$ and $j \in C$,
    \item the submatrix $H_2$ of $H$ determined by the rows from $R_1$ and the columns from $C_1$ has column sums equal to zero.
\end{enumerate}

Let $K$ be the matrix obtained from $H$ by multiplying the entries of $H_2$ by -1, i.e. by replacing $H_2$ with $-H_2$. Then $K$ is a Hadamard matrix.
\end{theorem}

\begin{remark}
The roles of rows and columns in the statement of Theorem \ref{des-had} can be interchanged and the claim holds.
\end{remark}

Note that $R_1$ and $C_1$ are switching sets of rows and columns, respectively, and this is a rank $1$ switching. Theorem \ref{des-had} above is a special case of the following theorem, describing a set of general conditions that, if satisfied, guarantee the existence of rank $1$ switchings of a complex Hadamard matrix.




\begin{theorem}\label{thm:rank1blocks}
Let $H$ be a complex Hadamard matrix, let $\{C_{1},\ldots,C_{t}\}$ be a partition of the set of columns $H$, and let $\{R,R_{1}\}$ be a partition of the set of rows of $H$, such that any two columns from distinct submatrices $M_{1i}$ and $M_{1j}$ defined by $R_{1}$ and $C_{i}$ and $C_{j}$ respectively with $i \neq j$ are pairwise orthogonal. Then any subset of $\{C_{1},\ldots, C_{t}\}$ is a switching set of columns of $H$, and $\{R_{1}\}$ is a switching set of rows of $H$. The switching coefficients $a_{1j}$ can be any number in $\mathbb{T}$ acting by multiplication. This is a rank $1$ switching.
\end{theorem}
\begin{proof}
    The inner product of any two columns of $H$ is preserved by multiplying any of the submatrices $M_{1i}$ by an element of $\mathbb{T}$. It is a rank $1$ switching because only one of the switching coefficients needs to be distinct from $1$.
\end{proof}

\begin{corollary}\label{cor:rank1blocks}
    Let $H$ be any complex Hadamard matrix of order $mn$ obtained by taking the Kronecker product $H_{m} \otimes H_{n}$ of complex Hadamard matrices $H_{m}$ and $H_{n}$ respectively. Then the columns of $H$ can be partitioned into sets $\{C_{1},\ldots, C_{n}\}$, with each $C_{i}$ of size $m$, and the rows can be partitioned into sets $\{R,R_{1}\}$ with $R_{1}$ of size $n$, such that the conditions of Theorem \ref{thm:rank1blocks} are satisfied. Hence there are infinitely many rank $1$ switchings of $H$.
\end{corollary}

\begin{proof}
Let $R_{1}$ be the first $n$ rows of $H$, comprised of $m$ scalar multiples of $H_{n}$. Let the columns of $C_{i}$ be the $i^{\rm th}$ columns of each such submatrix. Then the submatrices $M_{1i}$ are rank $1$ matrices of repeated columns for all $i$. Since $H_{n}$ is Hadamard, the columns from any pair of  submatrices $M_{1i}$ and $M_{1j}$ with $i \neq j$ are orthogonal.
\end{proof}

\begin{remark}
The switchings of Corollary \ref{cor:rank1blocks} are easily restricted to finitely many switchings of real and Butson Hadamard matrices by restricting the switching coefficients appropriately.
\end{remark}

Certain known affine families of parameterised Hadamard matrices can be described in the framework of matrices satisfiying the conditions of the Theorems of this section. The example $F_{4}(z)$ given in Section \ref{sec:paras} satisfies Theorem \ref{thm:rank1blocks}. 

\subsection{Conditions for rank $2$ switchings}

The conditions described below apply to irreducible rank $2$ switchings. The specific form of the matrices described in Section \ref{sec:GenHall} is generalised here. Our purpose is to extract the key conditions of that example which are sufficient to ensure a switching is possible.

\begin{theorem}\label{thm:genHall2}
Let $H=[h_{ij}]$ be a complex Hadamard matrix. Further, let $R$, $R_1$, $R_2$ be a partition of the rows of $H$, and let $C$, $C_1$, $C_2$ be a partition of the sets of columns of $H$, such that \begin{enumerate}
\item the submatrix of $H$ having the rows from $R$ and columns from $C$ have columns sums equal to $s$;
\item the submatrix determined by $R_1$ and $C_1$ have columns sums equal to $-\overline{s}$,
\item the submatrices of $H$ determined by $R$ and $C_1$ and by $R_1$ and $C$ are all-ones matrices,
\item the submatrices of $H$ determined by the columns from $C_2$ and the rows from $R$ and $R_{1}$, respectively have column sums equal to 0.
\end{enumerate}

The matrix $K$ obtained by multiplying the entries of the submatrix of $H$ determined by $R_1$ and $C_2$ by $z \in \mathbb{T}$, and simultaneously multiplying the entries of the submatrix determined by $R_2$ and $C_1$ by $\overline{z}$, is Hadamard.
\end{theorem}

\begin{proof}
Observe that the inner product of any two columns from $H$ is of the form $x + x_{1} + x_{2} = 0$ with $x$, $x_{1}$ and $x_{2}$ corresponding to the inner product of the components of the column with rows labeled by $R$, $R_{1}$ and $R_{2}$ respectively. Clearly the switching preserves the inner product of any pair of columns from within the same subset in the partition. It remains to consider the inner product of a pair of columns taken from any two distinct sets. We take each case separately.
\begin{itemize}
\item $C$ and $C_{1}$: By hypothesis $x + x_{1} = 0$ and so $x_{2} = 0$. The corresponding columns in $K$ have inner product $x + x_{1} + zx_{2} = 0$.
\item $C$ and $C_{2}$: By hypothesis $x_{1} = 0$ and so the corresponding columns in $K$ have inner product $x + zx_{1} + x_{2} = 0$.
\item $C_{1}$ and $C_{2}$: By hypothesis $x = 0$ and so $x_{1} + x_{2} = 0$. The corresponding columns in $K$ have inner product $x + \overline{z}(x_{1} + x_{2}) = 0$.
\end{itemize}
\end{proof}

\begin{corollary}\label{cor:col-row}
A matrix $H$ satisfying the conditions of Theorem \ref{thm:genHall2} also has the following equivalent set of conditions.
\begin{enumerate}[(i)]
\item the submatrix of $H$ having the rows from $R$ and columns from $C$ have row sums equal to $s$;
\item the submatrix determined by $R_1$ and $C_1$ have row sums equal to $-\overline{s}$,
\item the submatrices of $H$ determined by $R$ and $C_1$ and by $R_1$ and $C$ are all-ones matrices,
\item the submatrices of $H$ determined by the rows from $R_2$ and the columns from $C$ and $C_{1}$ respectively have row sums equal to $0$.
\end{enumerate}
\end{corollary}

\begin{proof}
Condition (iii) is already assumed by Condition 3 of Theorem \ref{thm:genHall2}. Now suppose that $K$ is the matrix obtained from switching as described in the theorem. Because $H$ and $K$ are complex Hadamard, the inner product of any two rows of the matrices is $0$. Observe that the inner product of any two rows from $H$ is of the form $y + y_{1} + y_{2} = 0$ with $y$, $y_{1}$ and $y_{2}$ corresponding to the inner product of the components of the rows with columns labeled by $C$, $C_{1}$ and $C_{2}$ respectively. 

Consider the inner product of a row from $R$ with a row from $R_{1}$. Since $K$ is also complex Hadamard, it is also true that $y + y_{1} + \overline{z}y_{2} = 0$, and consequently both $y_{2} = 0$ and $y + y_{1} = 0$. Let $r$ be the sum of the entries in the row corresponding to columns from $C$. Then $y = r$ and it follows that $y_{1} = -r$. Since the row from $R_{1}$ is arbitrary, every row sum in the submatrix defined by $R_{1}$ and $C_{1}$ must be $-\overline{r}$, and so is equal to $-\overline{s}$, and similarly, every row sum in the submatrix defined by $R$ and $C$ is $r$, and so is equal to $s$. This ensures Conditions (i) and (ii).

Next consider the inner product of a row from $R$ row from $R_{2}$. Since $K$ is also complex Hadamard, it is also true that $y + zy_{1} + y_{2} = 0$, and so $y_{1} = 0$. Since the submatrix defined by $R$ and $C_{1}$ is the all-ones matrix, it follows that the row sum of the row from the submatrix defined by $R_{2}$ and $C_{1}$ is $0$.

Finally, consider the inner product of a row from $R_{1}$ row from $R_{2}$. Since $K$ is also complex Hadamard, it is also true that $y + z(y_{1} + y_{2}) = 0$, and so $y = 0$. Since the submatrix defined by $R_{1}$ and $C$ is the all-ones matrix, it follows that the row sum of the row from the submatrix defined by $R_{2}$ and $C$ is $0$. Hence, Condition (iv) is satisfied.

This proves that the conditions of Theorem \ref{thm:genHall2} imply Conditions (i)-(iv). The converse is immediate by observing that the transpose of a matrix satisfying Conditions (i)-(iv) satisfies Conditions 1-4 of Theorem \ref{thm:genHall2}, and applying the same argument.
\end{proof}

The implication of Corollary \ref{cor:col-row} is that in a search for switching sets in the equivalence class of a complex Hadamard matrix, one can try to impose either set of conditions.

\subsection{Connections to parameterised families}

The example of an affine family of complex Hadamard matrices $D_{6}(z)$ given in Section \ref{sec:paras}, up to Hadamard equivalence, satisfies the conditions of Theorem \ref{thm:genHall2}. As a result, this family is contained in a single equivalence class up to rank $2$ switches. Petrescu's affine family introduced in \cite{Petrescu} given by
\[
P_{7}(z) = \left(\begin{array}{rrr|rr|rr}
1 & -1 & \omega^{5} & 1 & 1 & 1 & 1 \\
\omega^{5} & 1 & -1 & 1 & 1 & \omega^{2} & \omega^{2} \\
-1 & \omega^{5} & 1 & 1 & 1 & \omega^{4} & \omega^{4} \\ \hline
1 & 1 & 1 & \omega^{5} & -1 & \omega z & \omega^{4}z \\ 
1 & 1 & 1  & -1 & \omega^{5} & \omega^{4}z & \omega z \\ \hline
1 & \omega^{4} & \omega^{2} & \omega \overline{z} & \omega^{4}\overline{z} & \omega^{5} & -1 \\
1 & \omega^{4} & \omega^{2} & \omega^{4}\overline{z} & \omega \overline{z} & -1 & \omega^{5} 
\end{array}\right),
\]
where $\omega = \zeta_{6}$, also satisfy this theorem and the same applies.

A general construction by Di\c{t}\u{a} \cite{Dita} uses a tensor product like construction to build matrices of order $kn$ with numerous free parameters from a $k \times k$ complex Hadamard matrix $A$ and a set of $k$ possible different $n \times n$ matrices $\{B_{1},\ldots,B_{k}\}$. Then for diagonal unitary matrices $\{E_{2},\ldots,E_{k}\}$, the matrix
\[
H = \left(\begin{array}{cccc}
a_{11}B_{1} & a_{12}E_{2}B_{2} & \cdots & a_{1k}E_{K}B_{K} \\
\vdots & \vdots & \ddots & \vdots \\
a_{k1}B_{1} & a_{k2}E_{2}B_{2} & \cdots & a_{kk}E_{K}B_{K} \end{array}\right)
\]
is Hadamard. The entries on the diagonals of $E_{2},\ldots,E_{K}$ are free, with the exception of the entry in position $(1,1)$ which is fixed to be $1$ to preserve the dephased property. The matrix $H$ therefore has at least $(k-1)(n-1)$ free parameters, with further free parameters inherited from the matrices $A$ and $B_{1},\ldots,B_{k}$. Every choice of a value on a diagonal of each matrix $E_{2},\ldots,E_{K}$ corresponds to a rank $1$ switching set, and making multiple different choices corresponds to a reducible switching set.

\section{Applying switching to Butson matrices}\label{sec:generalize}

In this section, we give some examples of special cases of the results in the previous section. These demonstrate switchings for Butson matrices that generalise the methods of Orrick described in Section \ref{sec:Orrick}. 
Throughout this section switching coefficients will be $k^{\rm th}$ roots of unity acting by multiplication.

\subsection{Switching a closed set}

Let $H \in \mathrm{BH}(mk,k)$, and let $x = \zeta_{k}$ to ease notation. In some cases it is possible that up to Hadamard equivalence, i.e., equivalence under the action of $\mathrm{Mon}_{mk}(\langle x \rangle)^{2}$, the matrix $H$ can be arranged so that the first $k$ rows are of the form
$K \otimes j_{m}$ with $K \in \mathrm{BH}(k,k)$ and $\otimes$ denotes the Kronecker product operation. If so, a rank $1$ switching much like switching a closed quadruple in the real case is possible.

\begin{proposition}\label{prop:FourierSwitch}
Let $H \in \mathrm{BH}(mk,k)$ be a matrix such that the first $k$ rows are equal to $K \otimes j_{m}$ for some $K \in \mathrm{BH}(k,k)$. Then any matrix $H'$ obtained by multiplying any of the $k \times m$ blocks from the first $k$ rows of $H$ by any element of $\langle x \rangle$ is also in $\mathrm{BH}(mk,k)$.
\end{proposition}

\begin{proof}
This is a consequence of Theorem \ref{thm:rank1blocks} where $R_{1}$ labels the first $k$ rows and $\{C_{1},\ldots,C_{k}\}$ partitions the columns such that $C_{j}$ is the $j^{\rm th}$ block of $m$ columns. We also give a direct proof.
Let $(a_{1}\, a_{2}\, \cdots \, a_{k})$ be any row of $H$ below the $k^{\rm th}$ row, with each of the $a_{i}$ being of length $m$. Then the inner products of this row with each of the first $k$ rows must be zero, yielding the $k$ equations, which can be expressed in matrix form as
\[
K\left(\begin{array}{c} 
\langle j_{m},a_{1}\rangle\\
\langle j_{m},a_{2}\rangle\\
\vdots \\
\langle j_{m},a_{k}\rangle \end{array}\right) = \left(\begin{array}{c} 
0\\
0\\
\vdots \\
0 \end{array}\right).
\]
Since $K$ is invertible, it follows that $\langle j_{m},a_{i} \rangle = 0$ for all $1 \leq i \leq k$. Hence the inner product of this row with any of the first $k$ rows is not altered by multiplying any of the $k \times m$ blocks from the first $k$ rows of $H$ by any element of $\langle x \rangle$.
\end{proof}

It follows that we have $k$ different rank $1$ switching sets, and for each one, the switching coefficient can be any of $\langle x \rangle \setminus 1$.

\begin{corollary}\label{cor:p-restriction}
Let $p$ be prime and let $H \in \mathrm{BH}(mp,p)$ where $m$ is coprime to $p$. Then $H$ cannot have its first $p$ rows in the form $K \otimes j_{m}$ where $K \in \mathrm{BH}(p,p)$.
\end{corollary}

\begin{proof}
Suppose otherwise. Then a row $(a_{1}\, a_{2}\, \cdots \, a_{p})$ beneath the $p^{\rm th}$ row with each $a_{i}$ of length $m$ must be such that the entries of each $a_{i}$ sum to zero. Since $m$ is coprime to $p$, this is impossible.
\end{proof}

\subsection{Generalised switching of Hall sets}\label{sec:GenHall}

Let $x = \zeta_{k}$ to ease notation. Suppose there exists a circulant matrix $S = \mathrm{circ}(s_{0},\ldots,s_{k-1}) \in \mathrm{BH}(k,k)$. We recall some facts about circulant matrices from Davis \cite{Davis79}. The $k$ eigenvalues of $S$ are obtained by substituting the different $k^{\rm th}$ roots of unity into the function
\begin{equation}\label{eig eq}
f(z) = \textstyle{\sum}_{j=0}^{k-1}s_{j}z^{j}.
\end{equation}
Alternatively, the eigenvalues $\lambda_{0},\ldots,\lambda_{k-1}$ are the solutions to
\begin{equation}\label{Fourier}
(s_{0} \, \cdots \, s_{k-1}) F_{k} = (\lambda_{0} \, \cdots \, \lambda_{k-1}).
\end{equation} 
Hence, the first eigenvalue $\lambda_{0}$ of $S$ is $f(1) = \textstyle{\sum}_{j=0}^{k-1}s_{j}$. The eigenvalue $\lambda_{m} = c\lambda_{0}$ for some $c$ of norm $1$ since the eigenvalues of $S$ all have the same norm. Any row subsequent to the first row of $S$ is obtained by cyclically shifting the first row forward. Hence if we replaced the first row of $S$ in Equation \eqref{eig eq} with the $(j+1)^{\rm th}$ for some $j$, then substituting $x^{m}$ in for $z$ is equivalent to finding the $(m+1)^{\rm th}$ eigenvalue of $P_{j}S$, where $P_{j}$ is the circulant permutation matrix that pushes each row up $j$ places cyclically. The $(m+1)^{\rm th}$ eigenvalue of $P_{j}$ is $x^{mj}$. Because $P_{j}$ is circulant and commutes with $S$, they are simultaneously diagonalisable, and as a consequence we know that the $(m+1)^{\rm th}$ eigenvalue of $P_{j}S$ is $\lambda_{m} x^{mj}$.

We can now prove the following.

\begin{theorem}\label{thm:genHall}
Let $H \in \mathrm{BH}(kn+k,k)$ be of the form \[
H = \left(\begin{array}{cc}
S & F_{k} \otimes j_{n} \\
F_{k}^{\ast} \otimes j_{n}^{\top} & A \end{array}\right)
\]
where $S = \mathrm{circ}(s_{0},\ldots,s_{k-1})$ is a circulant matrix in $\mathrm{BH}(k,k)$ with eigenvalues $\lambda_{0},\ldots,\lambda_{k-1}$, and $A$ is a $kn \times kn$ block matrix $[A_{ij}]_{0 \leq i,j \leq k-1}$ comprised of $n \times n$ matrices $A_{ij}$. Then
\begin{itemize}
\item the entries in a row or column of $A_{ij}$ sum to zero if $i \neq j$;
\item the entries in a row or column of $A_{mm}$ sum to $-\overline{\lambda_{m}}$.
\end{itemize}

\end{theorem}

\begin{proof}
Let $r$ be any row between the $(k + mn +1)^{\rm th}$ and $(k + mn + n -1)^{\rm th}$ row of $H$. Then this row is of the form
\[
r = (1 \,~ \overline{x^{m}} \,~ \overline{x^{2m}} \,~ \cdots \,~ \overline{x^{m(k-1)}} \,~ \mid \,~ a_{1} \,~ a_{2}\,~ \cdots \,~ a_{k} )
\]
where $a_{j}$ denotes the section that is a row of $A_{mj}$. Let $\alpha_{j}$ denote the sum of the entries in $a_{j}$ and let $r_{1},\ldots, r_{k}$ be the first $k$ rows of $H$. Then, with reference to the discussion preceding the theorem statement, the $k$ inner products $\langle r_{j},r\rangle$ for $1 \leq j \leq k$ yield the equations
\begin{align*}
\lambda_{m} + \overline{\alpha_{1}} + \overline{\alpha_{2}} +\cdots +\overline{\alpha_{k}} &= 0 \\
\lambda_{m}x^{m} + \overline{\alpha_{1}} + x\overline{\alpha_{2}} +\cdots +x^{k-1}\overline{\alpha_{k}} &= 0 \\
&\vdots \\
\lambda_{{m}} x^{m(k-1)} + \overline{\alpha_{1}} + x^{k-1}\overline{\alpha_{2}} +\cdots + x\overline{\alpha_{k}} &= 0.
\end{align*}

Rearranging and writing as a matrix equation we can obtain
\[
F_{k}\left(\begin{array}{c}\overline{\alpha_{1}} \\ \overline{\alpha_{2}} \\ \vdots \\ \overline{\alpha_{k}} \end{array}\right) = -\left(\begin{array}{c}  \lambda_{m} \\ \lambda_{m} x^m \\ \vdots \\  \lambda_{m} x^{m(k-1)} \end{array}\right).
\]
Multiplying both sides on the left by $F_{k}^{-1}$ we obtain 
\[
\left(\begin{array}{c}\overline{\alpha_{1}} \\ \overline{\alpha_{2}} \\ \vdots \\ \overline{\alpha_{k}} \end{array}\right) = 
-\lambda_{m} e_{m},
\]
where $e_{m}$ is the $m^{\rm th}$ standard basis vector. Hence $\alpha_{m} = -\overline{\lambda_{m}}$ and $\alpha_{j} = 0$ for $j \neq m$, as claimed.
\end{proof}

\begin{corollary}\label{cor:genHall}
Suppose that $H \in \mathrm{BH}(nk+k,k)$ is in the form of Theorem \ref{thm:genHall}. Then let $R_{1}$ and $C_{1}$ be the first $k$ rows and columns respectively, and $R_{2}$ and $C_{2}$ be any set of the $n$ consecutive rows and columns beginning in position $k+mn + 1$ for some $m$. Then $\{R_{1},R_{2}\}$ and $\{C_{1},C_{2}\}$ are switching sets of size $2$, where the switching coefficients are $a_{12} = \omega$ for some $\omega \in \langle x \rangle \setminus 1$, $a_{21} = \overline{\omega}$ and $a_{11} = a_{22} = 1$. This is a rank $2$ switching.
\end{corollary}

\begin{proof}
Observe that by carrying out this switching, the inner product of any two rows from $R_{1}$ and $R_{2}$ is preserved, so the matrix remains Hadamard. 
\end{proof}

\begin{remark} 
A matrix satisfying the conditions of Theorem \ref{thm:genHall2} is in a less restricted form than that of Theorem \ref{thm:genHall}, however, it only describes a single rank $2$ switching set of rows and columns, rather than the range described in Corollary \ref{cor:genHall}. 
\end{remark}

In the same way that Hall sets allow us to switch real Hadamard matrices of order not divisible by $8$, this generalisation enables a switching of a $\mathrm{BH}(n,p)$ when $n$ is not divisible by $p^{2}$. The following example demonstrates this.

\begin{example}\label{ex:GenHallEx}
The following, in exponent form, is a $\mathrm{BH}(12,3)$ in the form above, where we can switch a generalised Hall set by multiplying any of the three non-initial rank $1$ blocks in the first row and column by any $\omega \in \langle x \rangle$ and $\overline{\omega}$ respectively that we choose:
\[
\left(\begin{array}{ccc|ccc|ccc|ccc}
0 & 1 & 0 & 0 & 0 & 0 & 0 & 0 & 0 & 0 & 0 & 0\\
0 & 0 & 1 & 0 & 0 & 0 & 1 & 1 & 1 & 2 & 2 & 2\\
1 & 0 & 0 & 0 & 0 & 0 & 2 & 2 & 2 & 1 & 1 & 1\\ \hline
0 & 0 & 0 & 1 & 2 & 1 & 1 & 0 & 2 & 2 & 1 & 0\\
0 & 0 & 0 & 2 & 1 & 1 & 2 & 1 & 0 & 1 & 0 & 2\\
0 & 0 & 0 & 1 & 1 & 2 & 0 & 2 & 1 & 0 & 2 & 1\\ \hline
0 & 2 & 1 & 1 & 2 & 0 & 2 & 2 & 0 & 1 & 2 & 0\\
0 & 2 & 1 & 0 & 1 & 2 & 2 & 0 & 2 & 2 & 0 & 1\\
0 & 2 & 1 & 2 & 0 & 1 & 0 & 2 & 2 & 0 & 1 & 2\\ \hline
0 & 1 & 2 & 1 & 0 & 2 & 2 & 0 & 1 & 1 & 1 & 2\\
0 & 1 & 2 & 0 & 2 & 1 & 0 & 1 & 2 & 1 & 2 & 1\\
0 & 1 & 2 & 2 & 1 & 0 & 1 & 2 & 0 & 2 & 1 & 1
\end{array}\right).
\]

Note here that the row sum of the circulant block in the top left is $\lambda = 2 + x$, and on the diagonal the row sums of $A_{11}$, $A_{22}$ and $A_{33}$ are $-\overline{\lambda} = \overline{(-2-x)} = \overline{(x + 2x^{2})} = 2x+x^{2}$, followed by $-x\overline{\lambda}$ and $-\overline{\lambda}$.
\end{example}

\begin{remark}
The requirement of the existence of a circulant element of $\mathrm{BH}(k,k)$ in Theorem \ref{thm:genHall} is always met when $k$ is odd, due to the existence of a circulant matrix that is Hadamard equivalent to the Fourier matrix, defined as $S = \mathrm{circ}([\frac{1}{\sqrt{k}}x^{\frac{j(j-1)}{2}} \; : \; 1 \leq j \leq n])$. The $3 \times 3$ submatrix in the top left block of the matrix of Example \ref{ex:GenHallEx} is of this form.    
\end{remark}

\begin{example}

We let $i = \sqrt{-1}$ and $j = -i$, and use $-$ as shorthand for $-1$. The matrix
\[
H = \left(\begin{array}{cccc|cc|cc|cc|cc}
1 & i & - & i & 1 & 1 & 1 & 1 & 1 & 1 & 1 & 1\\
i & 1 & i & - & 1 & 1 & i & i & - & - & j & j\\
- & i & 1 & i & 1 & 1 & - & - & 1 & 1 & - & -\\
i & - & i & 1 & 1 & 1 & j & j & - & - & i & i\\ \hline
1 & 1 & 1 & 1 & i & i & 1 & - & 1 & - & 1 & -\\
1 & 1 & 1 & 1 & i & i & - & 1 & - & 1 & - & 1\\ \hline
1 & j & - & i & 1 & - & - & - & 1 & - & - & 1\\
1 & j & - & i & - & 1 & - & - & - & 1 & 1 & -\\ \hline
1 & - & 1 & - & 1 & - & - & 1 & j & j & 1 & -\\
1 & - & 1 & - & - & 1 & 1 & - & j & j & - & 1\\ \hline
1 & i & - & j & 1 & - & 1 & - & - & 1 & - & -\\
1 & i & - & j & - & 1 & - & 1 & 1 & - & - & -\\ 
\end{array}\right)
\]
is in $\mathrm{BH}(12,4)$. Multiplying any of the $4 \times 2$ blocks in the first row of blocks by any $x \in \langle i \rangle$, and simultaneously multiplying the corresponding $2 \times 4$ block in the first column of blocks by $\overline{x}$ preserves the property of being Hadamard. This is a switching. Moreover, doing so produces an inequivalent $BH(12,4)$. Note that the switching in this example is obtainable from two rank $1$ switchings. One can multiply any of the $4 \times 2$ blocks in the first row of blocks, and the corresponding $2\times 2$ block on the diagonal, by any $x \in \langle i \rangle$ and retain the Hadamard property. One can act similarly on the corresponding rows, specifically by $\overline{x}$ to reproduce the switching above. This is not a typical property of a matrix in the form of Theorem \ref{thm:genHall}, as Example \ref{ex:GenHallEx} demonstrates.
\end{example}

\section{Trades in complex Hadamard matrices}\label{sec:trades}

Also motivated by Orrick's notion of switching, in \cite{POCWan} the authors study a very general version of switching, which they call switching a trade, using terminology derived from the literature on Latin squares. For a complex Hadamard matrix, they define a \emph{trade} to be a set of entries which can be altered to obtain a different complex Hadamard matrix. They use the word \emph{switch} to define the process of replacing a trace by a new set of entries to obtain another Hadamard matrix. No particular structure is assumed, the entries in a trade can be any set of entries in the matrix, and the switching can be the replacement of any of these entries with any other complex number of norm $1$. That being said, the paper focuses on the case where a switch is the process of multiplying the trade by a single complex number. In particular, they define a \emph{rectangular trade} to be one where the entries form a submatrix that can be switched by multiplying all entries by the same complex number $c \neq 1$ of norm $1$. In the language of this paper, a rectangular trade is a switching set of rows or columns of size $1$, and the process of switching is a rank $1$ switching where the switching coefficient acts by multiplication. In a clash of terminology, in \cite{POCWan} they define the size of a trade to be the number of entries, a quantity not well captured by our definition. They make the following observation.

\begin{theorem}[cf.~Theorem 2.1, \cite{POCWan}]
Let $H$ be a real Hadamard matrix of order $n$ and suppose there is a trade $T$ such that by negating the entries in $T$ the resulting matrix $H'$ is Hadamard. Then the size of $T$ is at least $n$.
\end{theorem}

They pose the following open question. Can there exist a trade of size less than $n$ in an $n \times n$ complex Hadamard matrix? We can answer this question in the negative here, building on their following Lemma.

\begin{lemma}[Lemma 2.2, \cite{POCWan}]\label{lem:floor}
Let $H$ be a complex Hadamard matrix of order $n$ and let $B$ be a set of $b$ columns of $H$. If $\alpha$ is a non-zero linear combination of the elements of $B$, then $\alpha$ has at least $\lceil\frac{n}{b}\rceil$ non-zero entries.
\end{lemma}

\begin{theorem}\label{thm:floorgen}
Let $H$ be a complex Hadamard matrix of order $n$. Any trade $T$ is of size at least $n$.
\end{theorem}

\begin{proof}
Suppose that $|T| < n$. In the set $R$ of rows intersecting $T$, let $b$ be the minimum number of entries in $T$ that a row in $R$ intersects, and let $r$ be one of the rows intersecting $T$ in $b$ positions. Clearly $b \geq 2$ as the switching must preserve orthogonality with every unaffected row, of which there is at least one by hypothesis. Since $|T| < n$ it follows that $|R| \leq \lfloor \frac{n-1}{b} \rfloor$. Denote the set of columns such that the corresponding entry of $r$ is in $T$ by $B$, and let $s_{B}$ denote the subrow of any row $s$ of $H$ comprising the entries from the columns in $B$. If $s$ is one of the rows not intersecting $T$, then switching the entries in $r_{B}$ to get $r_{B}'$ preserves orthogonality. It follows that $\langle s_{B},r_{B} \rangle = \langle s_{B},r_{B}' \rangle$, and thus 
\[
\sum_{i\in B}s_{i}(\overline{r_{i}}-\overline{r_{i}'}) = 0.
\]
This is true for all rows $s \not\in R$. Hence, there is a linear combination of the columns in $B$ with non-zero entries in at most $|R| = \lfloor \frac{n-1}{b} \rfloor$ positions, contradicting Lemma \ref{lem:floor}. Hence it cannot be true that $|T| < n$.
\end{proof}

\begin{corollary}
Any switching of a complex Hadamard matrix of order $n$ requires changing at least $n$ entries in the matrix.
\end{corollary}

The proof of Theorem \ref{thm:floorgen} confirms the following.

\begin{corollary}\label{cor:lincom}
Let $H$ be a complex Hadamard matrix. If a trade $T$ exists, i.e., if there is a switching of $H$, and the size of the smallest intersection of a row $r$ of $H$ with $T$ is $b$, then there is a non-trivial linear combination of $b$ columns with at least $n - \lfloor \frac{n-1}{b} \rfloor$ zeros.
\end{corollary}

Some of the work of \cite{POCWan} was motivated by applications to compressed sensing. Optimal complex Hadamard matrices for constructing compressed sensing matrices have the property that no linear combination of $t$ rows contains more than $t$ zeros \cite{Sensing}. The roles of rows and columns are interchangeable in Corollary \ref{cor:lincom}, and hence the existence of a non-trivial switching  of a complex Hadamard matrix $H$ means that $H$ is typically not optimal for this purpose, as $b < n - \lfloor \frac{n-1}{b} \rfloor$ for all $2 \leq b \leq n-2$. Note that if $b = 1$ then the switching will be degenerate as it would necessarily be multiplication of a row or column by a constant. Examples of when $b = n-1$ include applying the Galois automorphism to the Fourier matrices of order $n$ when $n$ is an odd prime, and these do appear to be optimal. A non-trivial linear combination of any $t$ of
the columns of a Fourier matrix of prime order vanishes on at most $t-1$ coordinates \cite{Sensing}. But examples of this type are otherwise rare. This prompts the following question with which we conclude this paper.

\begin{question}
Can we construct families of complex Hadamard matrices with the property that no switching set exists where the number $b$ in the notation of Corollary \ref{cor:lincom} is in the range $[2,n-2]$?
\end{question}

\vspace*{0.5cm}

\noindent {\bf Acknowledgement} \\
Dean Crnkovi\'c and Andrea \v Svob were supported by {\rm C}roatian Science Foundation under the projects HRZZ-IP-2022-10-4571 and HRZZ-UIP-2020-02-5713, and by European Union-NextGenerationEU under the project number uniri-iz-25-46-KonGeoGraGru.

\section{Statements and Declarations}

\subsection{Declaration of competing interests}

The authors declare no conflict of interest.

\vspace*{0.2cm}

\bibliographystyle{abbrv}
\flushleft{
\bibliography{MyBiblio}
}


\end{document}